\documentclass[reqno,a4paper,12pt]{amsart}
\usepackage{amsmath,amscd,amsfonts,amssymb}
\usepackage{mathrsfs,dsfont}

\def\R{\mathbb{R}}

\def\Q{\mathbb{Q}}
\def\Z{\mathbb{Z}}

\def\M{\mathfrak{M}}
\def\B{\mathfrak{B}}

\def\gam{\gamma}
\def\Gam{\Gamma}
\def\lam{\lambda}
\def\Lam{\Lambda}
\def\1{\mathds{1}}
\def\eps{\varepsilon}

\renewcommand\le{\leqslant}
\renewcommand\ge{\geqslant}
\renewcommand\leq{\leqslant}
\renewcommand\geq{\geqslant}

\renewcommand\hat{\widehat}

\newcommand{\ft}[1]{\widehat #1}
\newcommand\dotprod[2]{\langle #1 , #2 \rangle}
\newcommand\supp{\operatorname{supp}}
\newcommand\spec{\operatorname{spec}}

\newcommand\mes{\operatorname{mes}}
\newcommand\diam{\operatorname{diam}}

\theoremstyle{definition}

\newtheorem*{remark*} {Remark}

\theoremstyle{plain}
\newtheorem{theorem}{Theorem}
\newtheorem{lemma}{Lemma}
\newtheorem{proposition}[lemma]{Proposition}
\newtheorem{corollary}[lemma]{Corollary}
\newtheorem*{corollary*}{Corollary}
\newtheorem*{definition*}{Definition}
\newtheorem{definition}{Definition}

\newcommand{\lemref}[1]{Lemma~\ref{#1}}
\newcommand{\corref}[1]{Corollary~\ref{#1}}
\newcommand{\thmref}[1]{Theorem~\ref{#1}}

\newtheorem*{theorem-m}{Theorem M}

\theoremstyle{definition}

\newenvironment{enumerate-math}
{\begin{enumerate}
\addtolength{\itemsep}{5pt}
}
{\end{enumerate}}

\newenvironment{enumerate-math-abc}
{\begin{enumerate}
\addtolength{\itemsep}{5pt}
}
{\end{enumerate}}

\begin{document}

\title{Quasicrystals and Poisson's summation formula}
\author{Nir Lev}
\address{Department of Mathematics, Bar-Ilan University, Ramat-Gan 52900, Israel}
\email{levnir@math.biu.ac.il}

\author{Alexander Olevskii}
\address{School of Mathematical Sciences, Tel-Aviv University, Tel-Aviv 69978, Israel}
\email{olevskii@post.tau.ac.il}

\thanks{Both authors are partially supported by their respective Israel Science Foundation grants.}

\begin{abstract}
We characterize the measures on $\R$ which have both their support and spectrum uniformly
discrete. A similar result is obtained in $\R^n$ for positive measures.
\end{abstract}


\maketitle


\section{Introduction}

       The subject of this paper is the analysis of measures in $\R^n$
       with discrete support and  spectrum.
       This subject is often  discussed in the framework of 
       so-called Fourier quasi\-crystals, see J.\ C.\ Lagarias' survey  \cite{lag2}
       and the references therein.
       The name ``quasicrystals'' was inspired by an experimental
       discovery in the middle of 80's of non-periodic 
       atomic structures with diffraction patterns consisting
       of spots.

       Sometimes a Fourier quasi\-crystal is defined as a countable
       set $\Lam$  which supports an (infinite) pure point measure $\mu$,
       such that its Fourier transform $\ft\mu$ is also a pure point
       measure, see \cite{dys}.
       This definition is too wide, though, and includes examples where
       the support and spectrum are both everywhere dense sets.
       Usually the support $\Lam$ is assumed to be a uniformly
       discrete set (see e.g.\ \cite{bom}, \cite{cah}).

       The subject goes back to the classical Poisson summation formula:
       if $f$ is a function on $\R$ (satisfying some mild
       smoothness and decay conditions)  and $\ft{f}$ is its Fourier transform,
       then
\[
           \sum_{n\in \Z} f(n)
=           \sum_{n\in \Z} \ft{f}(n).
\]
       In other words, the measure
\[
                \mu=\sum_{n\in \Z} \delta_n
\]
       satisfies the equality
\[
                 \ft \mu = \mu.
\]

        There is also a multi-dimensional version of Poisson's formula.
        Let $L$ be a (full-rank) lattice in $\R^n$, and $L^*$ be the dual
        lattice.   Then
\[
                \big( \sum_{\lam \in L} \delta_\lam \big)\, \ft{\,} \; = \frac1{\det(L)}
                \sum_{s \in L^*} \delta_{s}.
\]
        By simple procedures -- shifts, multiplication on exponentials,
        and taking linear combinations -- one may get different forms
        of  this result.         In particular (for $n=1$) it includes the Cauchy-Ramanujan
formulas and more general ones due to V. Lin (see \cite[pp. 283--289]{gur}).

        However, there are Poisson-type formulas which
        cannot be obtained this way.
        In the one-dimensional case, the problem of which other discrete summation
        formulas may exist  was studied by J.-P.\ Kahane and
        S.\ Mandelbrojt \cite{kah}.

        An interesting example can be found in \cite[p. 265]{gui},  which involves weighted sums of $f$ and $\ft f$ at the nodes
        $\{\pm (n+\frac1{9})^{1/2}\}$ $(n=0,1,2,\dots)$. 
        This  summation
        formula is also deduced from Poisson's one, but in a more tricky way.
          Notice that in  contrast to the classical case, the nodes in this example do not lie in a uniformly discrete set.

        The cut-and-project method, applied to lattices in a generic 
        position, leads to an important class of quasicrystals --  the ``model sets''.  
       Y.\ Meyer \cite{mey1, mey2} discovered fundamental
        connections of these non-periodic sets to harmonic analysis.

           On the other hand, if $\mu$ is the sum of equal atoms
         along a discrete set $\Lam$  and $\ft\mu$ is a positive pure point
         measure, then $\Lam$ is just a lattice.
         A simple proof of this fact was given by A. C\'{o}rdoba \cite{cor1}.
         A more general situation, when the atoms take finitely
         many different values, was considered in \cite[p. 25]{mey0}, \cite{cor2}, \cite{kol}.
         These results 
         are based on the Helson-Cohen
         characterization of idempotent measures in locally compact
         abelian groups.

         There is a conjecture (see e.g.\ \cite[p.\ 79]{lag2}) that if the support and
         spectrum of a measure are both uniformly discrete sets,
         then the measure has a periodic structure,
         and the corresponding summation formula can be obtained from
         Poisson's one by the procedures mentioned above.

         The main goal of this paper is to prove this conjecture.
         In the one-dimensional case this is done in full generality, while in several
         dimensions -- for positive (or positive-definite) measures.
         Our results were outlined in \cite{lo}.

\section{Results}

 A set $\Lam \subset \R^n$ is called uniformly discrete (u.d.) if
\begin{equation}
\label{uddef}
                d(\Lam):=\inf_{\lam,\lam'\in\Lam, \lam\neq\lam'} |\lam-\lam'| > 0.
\end{equation}
          We consider a (complex) measure $\mu$ on $\R^n$ supported on a u.d.\ set $\Lam$:
\begin{equation}
\label{mes1}
              \mu = \sum_{\lam\in \Lam} \mu(\lam) \delta_\lam, \quad
           \mu(\lam) \neq 0, \quad d(\Lam)>0.
\end{equation}
Assume that $\mu$ is  a temperate distribution, and
          that its Fourier transform 
\[
               \ft{\mu}(x) := \sum_{\lam\in \Lam} \mu(\lam) e^{-2\pi i \dotprod{\lam}{x}} 
\]
         (in the sense of distributions)
          is also a measure, supported by a u.d.\ set $S$:
\begin{equation}
\label{mes2}
           \ft{\mu}=  \sum_ {s\in S}  \ft{\mu} (s) \delta_s, \quad
           \ft{\mu}(s) \neq 0,  \quad d(S) >0.   
\end{equation}
          The set $S$ is the spectrum of the measure $\mu$.

\begin{theorem}\label{thm1}
Let $\mu$ be a measure on $\R$ satisfying \eqref{mes1} and \eqref{mes2}.
                        Then the support $\Lam$ is contained in a finite union
                        of translates of a certain lattice.
                        The same is true for $S$ (with the dual lattice).
\end{theorem}

\begin{theorem}\label{thm2}
Let $\mu$ be a positive measure on $\R^n$, $n>1$, satisfying \eqref{mes1} and \eqref{mes2}.
 Then the conclusion of Theorem \ref{thm1} holds.
\end{theorem}

 The following proposition completes the results,
            describing the explicit form of $\mu$.

\begin{theorem}\label{thm3}
Let $\mu$ be a measure in $\R^n$, $n\geq 1$, satisfying \eqref{mes1} and \eqref{mes2}, and
such that $\Lam$ is contained in a finite union of translates of a lattice $L$. 
Then $\mu$ is of the form
\begin{equation}
\label{mes3}
\mu=\sum_{j=1}^N P_j \sum_{\lam \in L + \theta_j} \delta_\lam
\end{equation}
where  $\theta_j$ is a  vector in $\R^n$,
and  $P_j(x)$ is a trigonometric polynomial $(1\leq j\leq N)$.
\end{theorem}

By a trigonometric polynomial $P(x)$ on $\R^n$  we mean a finite linear combination
of exponentials $\exp{2\pi i \dotprod{\omega}{x}}$. 

The conclusion of \thmref{thm3} shows that $\mu$ can be obtained from
the measure  $\sum_{\lam \in L} \delta_\lam$ in       Poisson's summation formula
by a finite number of shifts, multiplication on exponentials,
        and taking linear combinations.

Conversely, one can easily see that every measure $\mu$ of the form \eqref{mes3} 
satisfies both \eqref{mes1} and \eqref{mes2}, since $\ft\mu$ is of the same form
(with the dual lattice).


\section{Preliminaries}

\subsection{Notation}
By
$\dotprod{\cdot}{\cdot}$ and $|\cdot|$ we  denote the Euclidean 
scalar product and norm in $\R^n$. The open ball of radius $r$ centered at the origin is denoted
 $B_r := \{x \in \R^n: |x|<r \}$.

A set $\Lambda \subset \R^n$ is uniformly discrete (u.d.) if it satisfies \eqref{uddef}.
The set $\Lambda$ is relatively dense if there is $R > 0$ such that every ball of
radius $R$ intersects $\Lam$. 

By a ``distribution'' we shall  mean a temperate distribution on $\R^n$ (see \cite{rud}).
By a  ``measure'' we mean a complex, locally finite measure (usually infinite) which is
also a temperate distribution. 
As usual $\delta_\lam$   is the Dirac measure at the point $\lambda$.

If $\alpha$ is a temperate distribution, and $\varphi$ is a Schwartz function on $\R^n$,
then $\dotprod{\alpha}{\varphi}$ will denote the action of $\alpha$ on ${\varphi}$.

The Fourier transform in $\R^n$ will be normalized as follows:
$$\ft \varphi (t)=\int_{\R^n} \varphi (x) \, e^{-2\pi i\langle t,x\rangle} dx.$$
If $\alpha$ is a temperate distribution then its Fourier transform $\ft\alpha$ is defined by 
$\dotprod{\ft\alpha}{\varphi} = \dotprod{\alpha}{\ft\varphi}$.

We denote by $\supp(\alpha)$ the support of the distribution $\alpha$, and by
$\spec (\alpha)$ the support of its Fourier transform $\ft\alpha$.

By  a (full-rank) lattice $L \subset \R^n$ we mean the image of $\Z^n$ under some
invertible linear transformation $T$. The determinant $\det(L)$ is equal to $|\det (T)|$.
The dual lattice $L^*$ is the set of all vectors $\lambda^*$ such that $\dotprod{\lambda}{\lambda^*}
 \in \Z$, $\lambda \in L$.

If $A$ is a set in $\R^n$ then $\# A$ is the number of elements in $A$,
$\mes(A)$  or $|A|$  denote the Lebesgue measure of $A$, $\diam(A)$
is the diameter of $A$, and $\1_A$ is the indicator function of $A$.
By $A+B$ and $A-B$ we denote the set of sums and set of differences of two sets
$A,B$ in $\R^n$.

\subsection{Measures}
We  will need a few simple facts about measures in $\R^n$.

\begin{lemma}\label{lemma14}
Let $\mu$ be a measure in $\mathbb R^n$ supported by a u.d.\ set $\Lambda.$ Then
$\mu$ is a temperate distribution if and only if
$$|\mu(\lambda)|\le C(1+|\lambda|^N),\quad \lambda\in\Lambda,$$
for some positive constants  $C$ and $N$.
\end{lemma}

This  can be  proved using standard arguments.

\begin{lemma}\label{lemma15}
Let $\mu$ be a measure in $\mathbb R^n$ satisfying \eqref{mes1} and \eqref{mes2}. Then
\begin{equation}
\label{eqtrbdd}
\sup_{\lam \in \Lam} |\mu(\lambda)| < \infty.
\end{equation}
\end{lemma}

\begin{proof}
Fix a Schwartz function $\varphi$ such that $\hat\varphi(0)=1$ and $\supp(\ft\varphi)\subset B_\delta$, where $\delta:=d(\Lam)>0$. Then

\begin{equation}
\label{eqmusud}
|\mu(\lambda)|=\Big| \int \hat\varphi(x-\lambda)\, d\mu(x)  \Big|
=\Big| \int  \varphi(t) \,e^{2\pi i\langle \lambda,t\rangle}\,  d\hat\mu(t)\Big|
\le \sum_{s\in S}|\varphi(s)|\,|\hat\mu(s)|.
\end{equation}
By Lemma \ref{lemma14} there are constants $C,N$ such that
$|\hat\mu(s)|\le C(1+|s|^N)$. Thus the sum on the right-hand
 side of \eqref{eqmusud}  converges, and this establishes \eqref{eqtrbdd}.
\end{proof}

\begin{lemma}\label{lemma19}
Let $\mu$ be a non-zero, positive measure in $\mathbb R^n.$ Then $0\in\spec(\mu).$
\end{lemma}

\begin{proof}
If not, there is $\delta>0$ such that the support of the distribution $\ft \mu$ is disjoint from $B_\delta$.
Choose a Schwartz function $\varphi$ such that $\supp(\varphi)\subset B_\delta$ and $\hat\varphi>0.$
Then
$$\int \hat\varphi \, d\mu=\langle \hat\mu,\varphi\rangle=0.$$
Hence $\hat\varphi\mu$ is a non-zero positive measure with zero total mass, a contradiction.
\end{proof}

\subsection{Densities}

We will use the classical concepts of lower and upper uniform
           density of a  set $\Lam$. The first one plays a central
           role in Beurling's sampling theory for entire functions
           of exponential type.  The second one was used by Kahane 
           and Beurling  in the interpolation problem.  
           Here are their definitions:
\begin{align*}
D^-(\Lambda) & := \liminf_{R\to\infty} \,\inf_{x\in\R^n} \,\frac{\#(\Lambda\cap (x+B_R))}{|B_R|},\\[8pt]
D^+(\Lambda)& :=\limsup_{R\to\infty} \,\sup_{x\in\R^n} \,\frac{\#(\Lambda\cap (x+B_R))}{|B_R|}.
\end{align*}

 We also need  the following version of density:
\[
D_\#(\Lambda) :=\liminf_{R\to\infty}\frac{\#(\Lambda\cap B_R)}{|B_R|}.
\]

Clearly we have $D^-(\Lam) \le D_\#(\Lam) \le D^+(\Lam)$.

            Notice that if $\Lam$ is a u.d.\ set then the densities above are
            finite, and that their values are  invariant under translation of $\Lam$.
 The last claim is obvious for $D^-$ and $D^+$, and is easy to check for $D_\#$.

\subsection{Sampling and interpolation}
Let $\Omega$ be a compact set in $\R^n$, whose boundary has Lebesgue measure zero. 
We denote by $\B(\Omega)$ the   Bernstein space
consisting of all bounded, continuous functions $f$ on $\R^n$ such that the distribution $\ft f$
is supported by $\Omega$. 

Let  $\Lam$ be a u.d.\ set in $\R^n$. One says that
\begin{enumerate-math}
\item
$\Lam$ is a \emph{sampling set} for $ \B(\Omega)$ if 
there is a constant $C = C(\Lam, \Omega)$ such that
$$\sup_{x \in \R^n} |f(x)|\le C \sup_{\lambda\in\Lambda}|f(\lambda)|, \quad f \in \B(\Omega);$$
\item
 $\Lam$ is an \emph{interpolation set} for $\B(\Omega)$ if for any bounded sequence
of complex numbers $\{c_\lam\}_{\lam\in\Lam}$, there exists some $f \in \B(\Omega)$ satisfying
$f(\lam)=c_\lam$ $(\lam\in\Lam)$.
\end{enumerate-math}

Landau proved in \cite{lan} that the classical
            density conditions for sampling and interpolation remain to be necessary
            in the more   general situation:
\emph{
\begin{enumerate-math}
\item
If $\Lam$ is a {sampling set} for $ \B(\Omega)$, then $D^-(\Lam) \geq \mes (\Omega)$;
\item
If $\Lam$ is an {interpolation set} for $\B(\Omega)$, then $D^+(\Lam) \leq \mes (\Omega)$.
\end{enumerate-math}
}
Actually, Landau  considered $L^2$ versions of the sampling and interpolation problems
(a simple proof  can be found in \cite{no}).
The above results for the   Bernstein space can be deduced e.g.\ as in \cite[Theorem 2.1]{ou2}.


\section{Spectral gaps}

\subsection{}
A measure (or a distribution) $\mu$ is said to have a \emph{spectral gap} of size $a > 0$ if 
the Fourier transform $\ft\mu$ vanishes on a ball of radius $a$.

In dimension one, there is a simple  condition which is necessary for a u.d.\ set $\Lam$
to support a measure with a spectral gap.

\begin{proposition}\label{prop1}
Let $\Lambda\subset\mathbb R$ be a u.d.\ set, $d(\Lambda) \geq \delta > 0$. 
Assume that $\Lambda$ supports a non-zero measure $\mu$,  such that
$\ft{\mu}$ vanishes on the open interval $(0,a)$ for some $a>0$. Then
$$D_{\#}(\Lambda)\ge c(a,\delta),$$
where $c(a,\delta) > 0$ depends on $a$ and $\delta$ only.
\end{proposition}

The proof given below is similar to the one used in \cite[pp. 1044--1045]{ou1}. It is based on the following

\begin{lemma}\label{lemma13}
Let $\Lambda$ be a finite set contained in $(-R,R)\setminus(-\delta, \delta)$, where $d(\Lambda) \geq \delta > 0$, $R\ge 1$,  
and let $a>0$. There is $c(a,\delta) >0$ such that if $(\#\Lambda)/(2R) <c(a,\delta)$ then one can find a Schwartz function
$\varphi$ with the following properties:
\[
\varphi(0)=1, \quad  \varphi(\lam)=0 \quad  (\lam\in \Lambda), \quad  \spec(\varphi) \subset (0,a), \quad 
\sup\limits_{|x|\ge R}| \varphi(x)| \le 1.
\]
\end{lemma}

\begin{proof}
It will be convenient to assume that the number of points in $\Lam$ is even (if not,
we may just add a point to $\Lam$). Let $n:=(\# \Lambda)/2$ and $\eps:=n/R$. Define the polynomial
\[
P(z):=\prod\limits_{\lambda\in\Lambda}\frac{z-e^{i\pi\lambda/R}}{1-e^{i\pi\lambda/R}} \; .
\]
Then  $P(1)=1$. We have
\[
\max_{|z|=1}|P(z)|\le \prod\limits_{\lambda\in\Lambda} \frac{2}{2\sin\left|\frac{\pi\lambda}{2R}\right|}
 \le \prod_{\lambda\in\Lambda}\frac{R}{|\lambda|}.
\]
The  right-hand side is maximized when $\Lambda$ is the set $\{j \delta : 1 \leq |j| \leq n\}$. Hence
\[
\max_{|z|=1}|P(z)|\le \frac{R^{2n}}{\delta^{2n} (n!)^2} \le \Big(\frac{eR}{\delta n}\Big)^{2n} =\Big(\frac{e}{\delta\varepsilon}\Big)^{2\varepsilon R}.
\]

Given $a>0$, we choose a  Schwartz function  $\psi$ satisfying
\[
\spec(\psi) \subset (0,a/4), \quad \psi(0)=1, \quad \gamma := \sup_{|x|\geq 1} |\psi(x)| < 1.
\]
Set
\begin{equation}
\label{deff1}
\varphi(x):=P(e^{i\pi x/R}) \cdot(\psi(x/R) )^{\lfloor R \rfloor + 1}\;.
\end{equation}
Then $\varphi$ is a Schwartz function, $\varphi(0)=1$, $\varphi(\lam)=0$ for $\lam\in \Lambda$. The spectrum of the first factor in 
\eqref{deff1} is contained in $[0, \eps]$, 
while the spectrum of the second factor is contained in $(0, a/2)$. Hence, if $\eps < a/2$ then 
$\spec(\varphi) \subset (0,a)$. Finally, we have
$$
\sup_{|x|\ge R}|\varphi(x)| \le  
\left[\gamma \left(\frac{e}{\delta \varepsilon}\right)^{2\varepsilon}\right]^R.$$
If $\varepsilon$ is sufficiently small (depending on $a,\delta$) then the expression in
square brackets is smaller than one. The lemma is therefore proved.
\end{proof}

\begin{proof}[Proof of Proposition \ref{prop1}]
It will be enough to prove the claim under the assumption  that $\mu$ is a finite measure. The general case may be easily reduced to this one by
multiplying $\mu$ on a Schwartz function $\varphi$, such that $|\varphi|>0$ and $\spec(\varphi) \subset (-a/2,0)$. 
Then $\varphi  \mu$ is a non-zero, finite measure (by Lemma \ref{lemma14}) supported by $\Lam$ and has a spectral gap
$(0, a/2)$.

Assume that $D_{\#}(\Lambda)<c(a,\delta)$, where $c(a,\delta)$ is given by Lemma
\ref{lemma13}. We  will show that this implies $\mu=0$. 
Observe that, by translating $\mu$ and $\Lam$, and since $D_{\#}(\Lambda-\lambda) = D_{\#}(\Lambda)$ for every $\lambda$,  it will be enough to
consider the case when $0\in\Lambda$ and to prove that $\mu(0)$ must be zero.

Choose a sequence $R_j\to\infty$ such that $(\#\Lambda_j)/(2R_j)<c(a,\delta)$, where
\[
  \Lam_j := \Lambda\cap(-R_j,R_j) \setminus \{0\},
\]
and let $\varphi_j$ be the function given by Lemma \ref{lemma13} with $\Lam = \Lam_j$ and $R=R_j$. Since
$\ft{\mu}$ vanishes on $(0, a)$ we have
\[
\int_{\R} \overline{\ft{\varphi}_j(t)} \, \ft{\mu}(t) \, dt =0.
\]
On the other hand,
\[
\int_{\R} \overline{\ft{\varphi}_j(t)} \, \ft{\mu}(t) \, dt = \int_{\R} \overline{\varphi_j(x)} \, d\mu(x) = \mu(0)+
\sum_{|\lam| \geq R_j}  \overline{\varphi_j(\lam)} \, \mu(\lam).
\]
It follows that
\[
|\mu(0)|\le 
\sum_{|\lam| \geq R_j}  |\mu(\lam)| \to 0 \quad (j\to\infty),
\]
hence $\mu(0)=0$.
\end{proof}

\subsection*{Remarks}
1. A similar result can be found in \cite[Proposition 7]{kah}.

2. In \cite{mp}
                 a complete characterization is given of u.d.\ sets in $\R$
                 which may support a finite measure with a
                 spectral gap of given size, in terms of the lower
                 Beurling-Malliavin density. It follows from this characterization that one may take $c(a,\delta)=a$
in Proposition \ref{prop1} (however    we do not use  this result).

\subsection{}  
The situation in the multi-dimensional case $(n>1)$ is different, and the existence
of a spectral gap  is not sufficient to make a conclusion about the density of
the support. As a simple example consider the set $\Lam = \Z \times \{0\}$
in $\R^2$, which has density zero, but which is the  support of  the measure
\[
\mu = \sum_{n\in\Z} (-1)^n \, \delta_{(n,0)}
\]
having  a spectral gap around the origin.

 However, if a u.d.\ set $\Lam$ supports a measure which has not just
 a spectral gap, but an \emph{isolated atom} in the spectrum,
               then the support must have positive density. More
               precisely, we have the following

\begin{lemma}\label{lemma16}
Let $\Lam$ be a u.d.\ set in $\R^n$. Assume that $\Lam$ supports a  measure $\mu$  satisfying
\eqref{eqtrbdd}, and such that $\spec(\mu) \cap B_a = \{0\}$ for some $a>0$. Then
$$D^-(\Lambda) \geq c(a,n),$$
where $c(a,n)>0$ depends on $a$ and $n$ only.
\end{lemma}

\begin{proof}
It is well-known that a distribution supported by the origin is a finite linear combination of derivatives
of $\delta_0$. But condition \eqref{eqtrbdd} ensures that the distribution $\ft\mu$ can only have order zero in a neighborhood
of the origin. Hence there is a non-zero complex number $w$ such that $\hat\mu=w\, \delta_0$ in $B_a$.
By multiplying $\mu$ on $1/w$ we may suppose that $w=1$. 

Fix a Schwartz function $\psi$, such that
$\supp(\hat\psi)\subset B_{a/2}$ and $\hat\psi=1$ in $B_{a/3}$.
For each $x\in\mathbb R^n$ define a measure $\nu_x$ by
$$\nu_x := \psi_x \, \mu, \quad \text{where}\quad \psi_x(y):=\psi(y-x).$$
Then we have the following properties:
\begin{enumerate-math}
\item
$\nu_x$ is supported by $\Lambda$;
\item
$\hat\nu_x(t)= (\ft\psi_x \ast \ft\mu)(t) =  e^{-2\pi i\langle x,t\rangle}$ in $B_{a/3}$;
\item
$\nu_x$ is a finite measure, and 
$\int|d\nu_x|\le C$ 
 for some constant $C$ not depending on $x$.
\end{enumerate-math}

Let $f$ be a function in the Bernstein space $\B(\Omega)$, where $\Omega := \{x : |x| \leq a/4\}$. 
Let $\varphi$ be  a Schwartz function
such that $\varphi(0)=1$ and $\spec(\varphi)$ is contained
in the open unit ball. Then  $f_\delta(x) := f(x) \varphi(\delta x)$ is a Schwartz function,
and $\spec(f_\delta) \subset B_{a/4+\delta}$. Hence
$$f_\delta(x)= \int \ft f_\delta(t) \, e^{2\pi i\langle x,t\rangle} \, dt
= \int \ft f_\delta(t) \, \overline{\hat\nu_x(t)} \, dt  =  \int f_\delta \, \overline{d\nu_x}\,.$$
Letting $\delta \to 0$ it follows (e.g.\ by the bounded convergence theorem) that
$$f(x)= \int f \, \overline{d\nu_x}\,,$$
and hence
$$|f(x)|\le C \sup_{\lambda\in\Lambda}|f(\lambda)|.$$
As this holds for any  $f \in \B(\Omega)$, we get that $\Lambda$ is a sampling set for $ \B(\Omega)$.
By Landau's theorem we therefore have $D^-(\Lambda)\ge \mes (\Omega) = c(a,n)$, and this proves the claim.
\end{proof}

\subsection{}  
\begin{lemma}\label{lem6}
Given  $a >0$ there is $R=R(a,n)$ such that, if a measure $\nu$
is supported by a u.d.\ set $Q$ in $\R^n$, $d(Q) > a$, and if $\ft{\nu}$ 
                              vanishes on a ball of radius $R$,
                              then $\nu = 0$.
\end{lemma}

\begin{proof} 
This follows from Ingham type theorems used in interpolation theory in $\R^n$.
Given  $a >0$  there is $R=R(a,n)$ such that if $Q$ is any u.d.\ set in $\mathbb R^n,$ $d(Q)>a,$ then $Q$ is an interpolation set for 
the Bernstein space $\B(\Omega)$, where $\Omega := \{x : |x| \leq R/2\}$
(see for example \cite{ou2}).

Let $\nu$ be a measure supported by  $Q$ and such that the distribution $\hat\nu$ vanishes on $B_R$
(there is no loss of generality in assuming that the ball is centered at the origin).
Given  $\lambda\in Q$ one can find  $f \in \B(\Omega)$ such that
$f(\lambda)=1$ and $f(\lambda')=0$ for any $\lambda'\in Q$, $\lambda'\ne\lambda$.
Let $\varphi(x) := f(x)\psi(x)$, where $\psi$ is  a Schwartz function
such that $\psi(\lambda)=1$ and $\spec(\psi) \subset B_{R/2}$.
Then $\varphi$ is a Schwartz function, satisfying
\[
\varphi(\lambda)=1, \quad \varphi(\lambda')=0 \quad (\lambda'\in Q, \; \lambda'\ne\lambda), \quad
\spec(\varphi) \subset B_R.
\]
It follows that
$$\nu(\lambda)=\int \overline \varphi \, d\nu=\langle \hat\nu, \overline{\ft \varphi} \rangle=0.$$
As this holds for any $\lambda\in Q$, we obtain $\nu=0$.
\end{proof}

\section{Delone and Meyer sets}
\label{secmey}

\subsection{}                  
We will need the following concepts of Delone and Meyer sets in $\R^n$.

\begin{definition}
$\Lam$ is called a Delone set if $\Lam$ is both a u.d.\ and relatively dense set.
\end{definition}

\begin{definition}
$\Lam$ is called a Meyer set if the following two conditions are satisfied:
\begin{enumerate-math}
\item
$\Lam$ is a Delone set;
\item
There is a finite set $F$ such that $\Lam-\Lam\subset \Lam+F$.
\end{enumerate-math}
\end{definition}

 Meyer \cite{mey1, mey2} discovered important  connections of this class of sets
to certain problems in harmonic analysis. In particular, to the characterization of classes
of almost-periodic functions with common almost-periods, and to the concepts of Pisot
and Salem numbers in algebraic number theory.

\subsection{}
Meyer observed that a Delone set $\Lam$ is a Meyer set if and only if $\Lam-\Lam-\Lam$ is u.d.\ (see \cite{mey2}).

                    Lagarias \cite{lag1} proved that if $\Lam$ is a Delone set
                    and $\Lam-\Lam$ is u.d.\ then $\Lam$ is a Meyer set.
              We need a stronger version of this result:
            
\begin{lemma}\label{lem4}
Let $\Lambda\subset \mathbb R^n$ be a Delone set, such that $D^+(\Lam-\Lam) < \infty$. Then
$\Lam$ is a Meyer set.
\end{lemma}

             The proof below follows Lagarias' argument,
             and simplifies it, basing  also on \cite{moo} (see also \cite{baake}).

\begin{proof}[Proof of \lemref{lem4}]
By translation we may assume that $0\in\Lambda$. We fix $R>0$ such that every ball
of radius $R$ intersects $\Lambda$.

Let $h\in \Lambda-\Lambda$. Then $h = y-x$ for some $x,y\in \Lambda$.
Choose a sequence $x_0,x_1,\dots,x_s$ such that $x_0=x$, $x_s = 0$, $|x_i-x_{i+1}|<R$.
Define $y_i=x_i+h$, then $y_0=y$, $y_s=h$, $|y_i-y_{i+1}|<R$.
Choose $p_i,q_i\in\Lambda$ such that $|p_i-x_i|<R$, $|q_i-y_i|<R$ $(0\leq i \leq s)$,
where $p_0=x$, $q_0=y$ and $p_s=0$ (recall that $0\in \Lambda$).
It follows that $p_i-p_{i+1}$ and $q_i-q_{i+1}$ belong to the finite set $F_1 := (\Lambda-\Lambda)\cap B_{3R}$.

Set $h_i:=q_i-p_i$. Then
\[
h_{i}-h_{i+1} =(q_i-q_{i+1})- ( p_i-p_{i+1})   \in F_2 := F_1-F_1.
\]
Also
\[
|h_i-h|=|(q_i-y_i)-(p_i-x_i)|<2R,
\]
hence
\[
h_i \in V(h) := (\Lambda-\Lambda) \cap (h + B_{2R}).
\]

Since $D^+(\Lam-\Lam) < \infty$, there is a constant $M$
independent of $h$  such that $\#V(h)\leq M$.
Thus in the sequence $h_0, h_1, \dots, h_s$ appear at most $M$ distinct values. Write
\[
h_0 - h_s = (h_0-h_1)+(h_1-h_2)+\cdots+(h_{s-1}-h_s).
\]
If some $h_i$ and $h_j$ $(i<j)$ admit the same value, then we may remove from the sum above
all the terms  $(h_k-h_{k+1})$, $i \leq k < j$. By removing all such ``cycles'' 
it follows  that $h_0 - h_s$ belongs to the finite set $F$ consisting of all vectors
which may be expressed as the sum of at most $M-1$ elements from $F_2$. Hence
\[
h = h_0 = h_0  + (q_s - h_s)  = q_s + (h_0 - h_s) \in \Lambda +F.
\]
This proves  that $\Lambda-\Lambda\subset \Lambda+F$, so $\Lambda$ is a Meyer set.
\end{proof}

\subsection{} 
Let $\Gamma$ be a lattice in $\R^{n+m}= \mathbb R^n\times \mathbb R^m$ $(m\ge 0)$, and let $p_1$ and $p_2$ denote the projections 
onto $\R^n$ and $\R^m$, respectively. We assume that the restriction of $p_1$ to $\Gamma$ is
injective, and that $p_2(\Gamma)$ is dense in $\R^m$. Let  $\Omega$ be a bounded set in $\mathbb R^m$.

\begin{definition}
Under the assumptions above, the set
\begin{equation}
\label{eqmod}
\M(\R^n\times  \R^m,\Gamma,\Omega):=
\{p_1(\gamma):\gamma\in\Gamma,p_2(\gamma)\in\Omega\},
\end{equation}
is  called the {model set} defined by $\Gam$ and $\Omega$.
\end{definition}

This construction is known as  ``cut-and-project''.

Remark that the case  $m=0$ is not excluded in the above definition. In this case
one should  understand $\mathbb R^m$ to be $\{0\}$, and  the model
set obtained is just a lattice in $\R^n$.

The following theorem \cite[Sections II.5, II.14]{mey1} gives a characterization of Meyer sets in terms of model sets (see also \cite{moo}).

  \begin{theorem-m}[Meyer]
  Let $\Lambda$ be  a Delone set in $\mathbb R^n$. Then the following are equivalent:
\begin{enumerate-math}
\item
$\Lam$ is a Meyer set;
\item
There exists a model set
$M$ and  a finite set  $F$ such that $\Lambda\subset M+F$.
\end{enumerate-math}
\end{theorem-m}

\subsection{}

\begin{lemma}\label{lemma21}
Let $M=\M(\mathbb R^n\times\mathbb R^m,\Gamma,\Omega)$ 
be a model set in $\mathbb R^n,$ and suppose  that the boundary of
$\Omega$ is a set of Lebesgue measure zero in $\R^m.$ Then 
$$D^-(M)=D^+(M)=\frac{\mes(\Omega)}{\det(\Gamma)}.$$
\end{lemma}

This fact is well-known, see for example \cite[Proposition 5.1]{meymat}.

\subsection{}

For a set $A \subset \R^n$ we shall denote by  $\Z[A]$ the additive group generated by the elements of $A$.

\begin{lemma}\label{lemmaindep}
 Let $M= \M(\R^n\times  \R^m,\Gamma,\Omega)$ be a model set, and $F$ be a finite set in $\R^n$. Then
there is another model set  $M'= \M(\R^n\times  \R^m,\Gamma',\Omega')$  and  a finite set $F'$, such that
\[M+F \subset M'+F', \quad p_1(\Gamma')\cap \mathbb Z[F']=\{0\}, \quad \Gamma \subset \Gamma'.\]
  \end{lemma}

  \begin{proof}
The elements of $F$ generate a finite-dimensional vector space over the rationals $\Q$, which we denote
by $V=\Q[F]$. Let
  $U:=V\cap \mathbb Q[p_1(\Gamma)]$,  a linear subspace of $V$.
Let  $W$ be any linear subspace of $V$ such that $U\oplus W = V$.

Denote by $\theta_1, \dots, \theta_s$ the elements of $F$. Then each $\theta_j$ admits a unique representation as
$\theta_j=u_j+w_j$, where $u_j\in U$, $w_j\in W$. Since $U\subset \Q[p_1(\Gamma)]$ we may
find a non-zero integer $q$ and elements $\gamma_1,\dots,\gamma_s \in \Gamma$ such that
$u_j = p_1(\gamma_j / q)$, $1\leq j \leq s$. Define
\[
\Gamma':=(1/q)\Gamma, \quad \Omega':=\bigcup_{j=1}^s(\Omega+p_2(\gamma_j/q)), \quad
F':=\{w_1,\dots,w_s\}.
\]
Then  $\Gamma'$ is a lattice in $\R^n\times \R^m$,  the restriction of $p_1$ to ${\Gamma'}$ is  injective, and $p_2(\Gamma')$ is
dense in $\R^m$. The set $\Omega'$ is a bounded set in $\R^m$, and $F'$ is a finite set in $\R^n$.

Let
 $M'$ be the model set defined by
$\Gamma'$ and $\Omega'$.
We show that $M+F\subset M'+F'$. Indeed, an element  $\lambda \in M+F$ is of the form
$\lambda=p_1(\gamma)+\theta_j$, where  $\gamma\in\Gamma$ and $p_2(\gamma)\in \Omega$.
Set $\gamma' := \gamma + \gamma_j/q$, then $\gamma' \in \Gamma'$ and
$p_2(\gamma') \in\Omega'$. Hence
\[\lambda = p_1(\gamma')+w_j \in M'+F'.\]

Finally, observe that the set $p_1(\Gamma')\cap \mathbb Z[F']$ must be equal to $\{0\}$, since it
 is contained in both  $U$ and $W$. It is also clear that $\Gamma \subset \Gamma'$, and so the lemma is proved.
 \end{proof}

Notice that in the special case when  $m=0$, Lemma \ref{lemmaindep} reduces to:

\begin{corollary}\label{corollarylat}
 Let $L$ be a lattice, and $F$ be a finite set in $\R^n$. Then
there is another lattice $L'$ and  a finite set $F'$, such that
$L+F \subset L'+F'$, $L'\cap \mathbb Z[F'] =\{0\}$, $L \subset L'.$
\end{corollary}

\section{Proof of Theorems \ref{thm1} and \ref{thm2}}

\subsection{}
               We will use the following notation: for
                 $h\in \Lam-\Lam$, denote
\[
                     \Lam_h:=\Lambda\cap(\Lambda-h)=\{\lam \in \Lam : \lam+h \in \Lam\}.
\]
Clearly  $\Lam_h$ is  a non-empty subset of $\Lam$.

Let $\mu$ be a measure in $\R^n$ satisfying \eqref{mes1} and \eqref{mes2}.
For each $h \in \Lam-\Lam$ we introduce a new measure
\begin{equation}
\label{eqmuh}
\mu_h:=\sum_{\lambda\in\Lambda_h}\mu(\lambda)\,\overline{\mu(\lambda+h)}\,\delta_\lambda.
\end{equation}
Clearly it is a non-zero measure with   $\supp(\mu_h) = \Lam_h$ and with bounded atoms
(by Lemma \ref{lemma15}), so it is a temperate distribution.

\begin{lemma}\label{lem1}
Let $a := d(S)>0$. Then we have 
$\spec(\mu_h) \cap B_a \subset \{0\}$, that is,
the punctured ball $B_a \setminus \{0\}$ is free from
the spectrum of the measure  $\mu_h$.
\end{lemma}

\begin{proof}
We fix a Schwartz function $\varphi$ on $\R^n$, such that $\varphi(0)=1$, and whose spectrum is contained
in the open unit ball. Denote $\varphi_\delta(x) := \varphi(\delta x)$.

Let $u\in\mathbb R^n.$ Consider the measure
\begin{equation}\label{4}
(\hat \varphi_\delta\ast\hat\mu)(t+u)\cdot\overline{\hat\mu(t)}.
\end{equation}
It is a temperate distribution, supported by the set $S\cap(S-u+B_\delta).$ Hence, if
$$u\in U_\delta:=\mathbb R^n\setminus[(S-S)+B_\delta],$$
then the measure in \eqref{4} vanishes identically.

Now consider the Fourier transform of the measure \eqref{4}. It is the measure
\begin{align*}
&
[e^{2\pi i\langle u,x\rangle}\varphi_\delta(-x)\mu(-x)]\ast\overline{\mu(x)} = \sum_{\lambda\in\Lambda}\sum_{\lambda'\in \Lambda} e^{-2\pi i\langle u,\lambda\rangle}\varphi_\delta(\lambda)\mu(\lambda)\overline{\mu(\lambda')}\delta_{\lambda'-\lambda}\\
&=\sum_{h\in\Lambda-\Lambda}\left[\sum_{\lambda\in\Lambda_h}e^{-2\pi i\langle u,\lambda\rangle}\varphi_\delta(\lambda)\mu(\lambda)\overline{\mu(\lambda+h)}\right]\delta_h\\
&=\sum_{h\in\Lambda-\Lambda}(\varphi_\delta\cdot\mu_h) \,\ft{\,}\, (u)\cdot\delta_h.
\end{align*}
 It follows that
 for every $h\in\Lam-\Lam$ we have
 $$(\varphi_\delta \cdot \mu_h)\,\ft{\,}\, (u)=0,\quad u\in U_\delta.$$

The finite measure $\varphi_\delta\cdot\mu_h$ tends to $\mu_h$ (in the sense of temperate distributions) as $\delta\to 0$.
This implies that $\spec(\mu_h)$ is contained in the closure of
the set $S-S$,  which is disjoint from $B_a \setminus \{0\}$. The lemma is therefore proved.
\end{proof}

\begin{remark*}
If $\mu$ is a \emph{positive} measure, then so is $\mu_h$. Hence in this case Lemmas  \ref{lemma19} and \ref{lem1} imply that the distribution
                       $\ft\mu_h$ has an isolated atom at the origin.
\end{remark*}

\subsection{}  

\begin{lemma}\label{lem3}
Let $\Lam$ be a u.d.\ set in $\R^n$. Suppose there is $c=c(\Lam)>0$ such that
$D_\#(\Lambda_h) > c$ for every $h\in\Lambda-\Lambda$. Then $D^+(\Lam-\Lam)<\infty$.           
\end{lemma}

\begin{proof}
Let $x \in \R^n$. Suppose that $h_1,\dots,h_N$ are distinct vectors belonging to the set
$( \Lambda-\Lambda)\cap (x+B_\delta)$, where  $\delta := d(\Lam)/2 > 0$. 
If $\lam\in \Lam_{h_i} \cap \Lam_{h_j}$ $(i\neq j)$ then
\[
h_i - h_j = (\lam +h_i) - (\lam +h_j) \in ( \Lambda-\Lambda) \cap B_{2\delta} = \{0\},
\]
which is not possible. Hence $\Lambda_{h_1},\dots,\Lambda_{h_N}$ are pairwise disjoint subsets of $\Lam$.
Since the density $D_\#$ is  super-additive, it follows that
\[
D_\#(\Lambda) \ge \sum_{j=1}^{N} D_\#(\Lam_{h_j}) \ge c N.
\]
This shows that the set $\Lam-\Lam$ cannot have  more than $D_\#(\Lambda) /c$  elements in any ball of radius $\delta$,
thus  $D^+(\Lam-\Lam)<\infty$.
\end{proof}

\subsection{}

\begin{lemma}\label{lemma10} Let
$E$ be a bounded set in $\mathbb R^m$, and let $\xi$ be a vector in $E-E$ such that
$$|\xi|^2 > (\diam E)^2-\delta^2$$
for some $\delta>0$. Suppose that we are given two representations of $\xi$ as the difference of two elements from $E$:
\[
\xi=y_1-x_1=y_2-x_2, \qquad x_1,y_1,x_2,y_2\in E.
\]
Then $|x_1-x_2|<\delta$.
\end{lemma}

\begin{proof}
By the parallelogram law we have
$$|\xi|^2+|x_1-x_2|^2=\frac{1}{2}\,(|y_1-x_2|^2+|y_2-x_1|^2)\le(\diam E)^2,$$
so the claim follows.
\end{proof}

\subsection{}

\begin{lemma}\label{lem5}
Let  $\Lam$ be a Meyer set in $\R^n$.   Suppose there is $c=c(\Lam)>0$ such that
\begin{equation}
\label{eq6}
D^+(\Lam_h)>c
\end{equation}
 for every $h\in\Lambda-\Lambda$. Then $\Lam$ is contained in a finite union of translates of some lattice.
\end{lemma}

\begin{proof}
(i)
By  Theorem M there exists a model set
$M= \M(\R^n\times  \R^m,\Gamma,\Omega)$ and  a finite set  $F$ such that $\Lambda\subset M+F$.
By Lemma \ref{lemmaindep} we may suppose that
\begin{equation}
\label{eqindp}
p_1(\Gamma)\cap \mathbb Z[F]=\{0\}.
\end{equation}
Thus each $\lambda \in \Lam$ admits a unique representation as
\begin{equation}
\label{equnqrep}
\lambda=p_1(\gamma(\lambda))+\theta(\lambda),\qquad \gamma(\lambda)\in\Gamma,\ p_2(\gamma(\lambda))\in\Omega,\ \theta(\lambda)\in F.
\end{equation}
The uniqueness follows from \eqref{eqindp} and the fact that the restriction of $p_1$ to $\Gamma $ is injective.

\medskip

(ii)
Let $h\in \Lam-\Lam$, and suppose that $\lam_1,\lam_2\in\Lam_h$.
Denote
$$\lam'_j := \lam_j +h, \quad j=1,2.$$
Then from \eqref{equnqrep} we have
$$h=\lam'_j-\lambda_j=p_1(\gam(\lam'_j)-\gam(\lam_j))+(\theta(\lam'_j)-\theta(\lam_j)), \quad j=1,2.$$
The condition \eqref{eqindp} implies that the representation of $h$ as the sum of an element from
$p_1(\Gamma)$ and an element from $F-F$ is unique.
Hence, we must have
$$p_1(\gam(\lam'_1)-\gam(\lam_1)) = p_1(\gam(\lam'_2)-\gam(\lam_2)).$$
Since the restriction of $p_1$ to $\Gamma $ is injective, this implies
$$\gam(\lam'_1)-\gam(\lam_1) = \gam(\lam'_2)-\gam(\lam_2).$$
We thus obtain the following: to each  $h\in \Lam-\Lam$ there corresponds an element $H(h) \in \Gam$ such that
\begin{equation}
\label{eqHh}
\gam(\lam+h)-\gam(\lam) = H(h), \quad \lam \in\Lam_h.
\end{equation}

\medskip

(iii)
Let $E:=\{p_2(\gamma(\lambda)):\lambda\in\Lambda\}$. Then
$E$ is a bounded set in $\mathbb R^m$, $E\subset \Omega$.  
Given $\delta>0,$ we may choose a vector $\xi \in E-E$  such that
$|\xi|^2 > (\diam E)^2-\delta^2$.
Observe that
$$E-E = \{p_2(H(h)): h\in\Lam-\Lambda\},$$
hence $\xi = p_2(H(h))$ for some $h\in\Lam-\Lambda$.
Let us fix such an $h$.

Now suppose that $\lambda_1,\lam_2 \in\Lambda_h$. Then by \eqref{eqHh} we have
\[
H(h) = \gam(\lam_j+h)-\gam(\lam_j) , \quad j=1,2.
\]
This yields  two representations of $\xi$ as the difference of two elements from $E$:
\[
\xi = p_2(H(h)) = p_2(\gam(\lam_j+h))-p_2(\gam(\lam_j)), \quad j=1,2.
\]
By Lemma \ref{lemma10}  we must therefore have
\[
|p_2(\gam(\lam_2)) - p_2(\gam(\lam_1))| < \delta.
\]
Hence, we conclude the following: denote
\begin{equation}
\label{eqEh}
E(h) :=\{p_2(\gamma(\lambda)):\lambda\in\Lambda_h\}.
\end{equation}
Then, given any $\delta > 0$ one can find  $h \in \Lam-\Lam$ such that $\diam(E(h)) < \delta$.

\medskip

(iv)
Let $h \in \Lam-\Lam$, and suppose  that $\diam(E(h)) < \delta$ for some $\delta > 0$.
We may find an open ball $\Omega'$ of radius $\delta$ such that $E(h) \subset \Omega'$.
Consider the model  set 
$$M'= \M(\R^n\times  \R^m,\Gamma,\Omega').$$
Then by \eqref{eqmod}, \eqref{equnqrep} and \eqref{eqEh} we have $\Lam_h \subset M'+F$.
Since the density $D^+$ is  sub-additive and invariant under translations, this implies
\[
D^+(\Lam_h ) \leq \# F \cdot D^+(M').
\]
Recall that $D^+(M') =  (\det \Gamma)^{-1} |\Omega'|$, according to \lemref{lemma21}. Hence
\[
D^+(\Lam_h ) \leq \# F \cdot \frac{c_m  \delta^m}{\det \Gamma},
\]
where $c_m$ denotes the volume of the unit ball in $\R^m$.

\medskip

(v)
It follows from (iii),(iv) that if  $m \geq 1$, then we may find elements  $h \in \Lam-\Lam$ with
$D^+(\Lam_h)$ arbitrarily small, in contradiction to \eqref{eq6}. Hence we must have $m=0$,
that is, $M$ must be a lattice. Thus $M+F$ is a finite union of translates of a lattice. Since
$\Lambda\subset M+F$, this concludes the proof.
\end{proof}

\subsection{}
Now we can finish the proof of Theorems \ref{thm1} and \ref{thm2}.

\begin{proof}[Proof of Theorems \ref{thm1} and \ref{thm2}]
For each $h\in\Lambda-\Lambda,$ let $\mu_h$ be the measure defined by \eqref{eqmuh}.
Then $\mu_h$ is a non-zero measure, $\supp(\mu_h)=\Lambda_h$, and
$\sup_\lam |\mu_h(\lam)| < \infty$ (by \lemref{lemma15}).

By \lemref{lem1} we have
\begin{equation}\label{1}
\spec(\mu_h) \cap B_a \subset \{0\},
\end{equation}
where $a=d(S)>0.$

In the one-dimensional case  $n=1$, observe that
condition \eqref{1} implies that $\hat\mu_h$ vanishes on the open interval $(0,a).$
So we may use Proposition \ref{prop1} which gives
\begin{equation}\label{2}
D_\#(\Lambda_h)\ge c,\quad h\in\Lambda-\Lambda,
\end{equation}
where $c>0$ is a constant which depends on $d(\Lambda)$ and $d(S).$

In the multi-dimensional case $n>1,$ we use the extra assumption that $\mu$ is a positive measure. It implies that $\mu_h$ is also positive, for every $h\in\Lambda-\Lambda.$ 
By \lemref{lemma19} we  therefore have $0\in\spec(\mu_h),$ so 
$\spec(\mu_h) \cap B_a = \{0\}$.
This allows us to use \lemref{lemma16}, which gives that $D^-(\Lambda_h)\ge c,$ where $c>0$ is a constant which now depends on $d(S)$ only. 
Since $D_\#(\Lambda_h)\ge D^-(\Lambda_h)$, we obtain \eqref{2} again.

With \eqref{2} established, we now proceed to apply \lemref{lem3} which gives
\begin{equation}\label{3}
D^+(\Lambda-\Lambda)<\infty.
\end{equation}

Also, using \lemref{lem6} with $Q=S$ and $\nu=\hat\mu$ gives that $\Lambda$ is
a relatively dense set. Hence $\Lambda$ is a Delone set
(see also \cite[Lemma 1]{cor2}).

 This together with \eqref{3} gives, by \lemref{lem4}, that $\Lambda$ is a Meyer set.
 
 Finally, we apply \lemref{lem5}. Since from  \eqref{2} we get
 $D^+(\Lambda_h)\ge  c$ for every $h\in\Lambda-\Lambda,$ the lemma gives
 that $\Lambda $ is contained in a finite union of translates of some lattice,
and this completes the proof.
 \end{proof}

\section{Proof of Theorem  \ref{thm3}}

\subsection{}

\begin{lemma}\label{lem24}
Let $\theta\in\mathbb R^n\setminus \mathbb Q^n.$ Then the set 
$$H(\theta):=\{m \in\mathbb Z^n: \langle\theta,m\rangle\in\mathbb Z\}$$
is contained in some $(n-1)$-dimensional hyperplane.
\end{lemma}

\begin{proof}
Define
$V(\theta):=\{x\in\mathbb Q^n:\langle \theta,x\rangle\in\mathbb Q\}.$
It is a linear subspace of $\mathbb Q^n$ over the rationals. Since $\theta\notin\mathbb Q^n,$ this subspace cannot contain all the standard basis vectors $e_1,\dots,e_n$. Hence $V(\theta)$ is a proper subspace of $\Q^n$, and so it is necessarily contained in some $(n-1)$-dimensional hyperplane. But
$H(\theta)\subset V(\theta)$, so this proves the claim.
\end{proof}

Since the union of a finite number of hyperplanes cannot cover $\mathbb Z^n$, it follows that:

\begin{corollary}\label{corollary25}
Let $\theta_1,\dots,\theta_s\in\mathbb R^n\setminus \mathbb Q^n.$ Then there is $m\in\mathbb Z^n$ such that
$$\langle\theta_j,m\rangle\notin\mathbb Z,\quad 1\le j\le s.$$
\end{corollary}

\subsection{}

\begin{proof}[Proof of Theorem \ref{thm3}]
We suppose that $\mu$ is a measure in $\mathbb R^n$ $(n\ge1)$
satisfying \eqref{mes1} and \eqref{mes2}, and that the support of $\mu$ is contained in a finite union of
translates of a lattice $L$.

Using Corollary \ref{corollarylat} we can find a larger  lattice $L' \supset L$ and a finite set $F'$
such that $L'\cap \mathbb Z[F']=\{0\}$, and the support of $\mu$ is contained in $L' + F'$.
We will show that $\mu$ can be represented in the form \eqref{mes3} with the lattice $L'$.
The desired representation with the original lattice  can  be obtained
by covering $L'$ with a finite number of translates of $L$.

It will be enough, by applying a linear transformation, to consider the case $L' = \mathbb Z^n.$

Denote by $\theta_1,\dots,\theta_s$ the elements of $F'$. For each $j=1,\dots,s$ define a measure
$$\mu_j:=\sum_{k\in\mathbb Z^n}\mu(k+\theta_j)\,\delta_k.$$
It is a temperate distribution (by Lemma \ref{lemma15})  supported by $\mathbb Z^n,$ and we  have
\begin{equation}\label{5}
\mu(x)=\sum_{j=1}^s\mu_j(x-\theta_j).
\end{equation}

The Fourier transform $\hat\mu_j$ is a temperate distribution on $\mathbb R^n$ which is $\mathbb Z^n$-periodic.
Define a distribution
\begin{equation}\label{eq5p}
\alpha_j(t):=e^{-2\pi i\langle\theta_j,t \rangle} \, \hat \mu_j(t).
\end{equation}
From \eqref{5}, \eqref{eq5p} and the periodicity of $\hat\mu_j$ it follows that
\begin{equation}\label{6}
\hat \mu(t-k)=\sum_{j=1}^s e^{2\pi i\langle\theta_j,k\rangle}\alpha_j(t)
\end{equation}
for each $k\in\mathbb Z^n$.

Since $\mathbb Z^n\cap \mathbb Z[\theta_1,\dots,\theta_s]=\{0\},$ we have
$$\theta_j-\theta_\ell\notin\mathbb Q^n\qquad (j\ne\ell).$$
Using \corref{corollary25}, we may therefore choose a vector $m\in\mathbb Z^n$ such that
\begin{equation}\label{7}
\langle\theta_j-\theta_\ell,m\rangle\notin\mathbb Z\qquad (j\ne\ell).
\end{equation}
Applying \eqref{6} with $k=p m$ $(p=0,1,2,\dots,s-1)$ yields a system of $s$ linear equations,
with a Vandermonde determinant that does not vanish
due to \eqref{7}. Hence this linear system may be inverted, and we obtain that
$$\alpha_j(t)=\sum_{p=0}^{s-1} c_{jp}\,\hat\mu(t-pm)$$
for appropriate coefficients $\{c_{jp}\}.$

But now using \eqref{eq5p} this implies that the distribution $\hat\mu_j$ is  a measure, supported by the closed, discrete set
$S+\{0,m,2m,\dots,(s-1)m\}$.
On the other hand, the measure $\hat\mu_j$ is $\mathbb Z^n$-periodic. Hence it must be of the form
$$\hat\mu_j= \nu_j \ast \sum_{k\in\mathbb Z^n}\delta_{k},$$
where $\nu_j$ is a measure which is a finite sum of point masses.
It follows that
$$\mu_j(x)=P_j(x)\sum_{k\in\mathbb Z^n}\delta_k$$
where  $P_j$ is a trigonometric polynomial, $P_j(x) = \ft \nu_j(-x)$. By \eqref{5} this completes the proof of Theorem \ref{thm3}.
\end{proof}

\section{Remarks}

1. Theorems \ref{thm1} and \ref{thm2} give an affirmative answer to Problem 4.1(a) in \cite[p.\ 79]{lag2}.
              The Problem 4.1(b) from that paper, asking whether
             one can remove the uniformity requirement
              for discrete sets $\Lam$ and $S$, remains open.
              For signed (not positive) measures, 
one may expect a counter-example 
          due to the results  in \cite{gui}.

We also leave open the problem whether \thmref{thm2} holds  for non-positive measures.
In \cite{lo} we proved this under the additional assumption that $S-S$ is a u.d.\ set.

            2. It is well-known that if one requires from $S$ in Theorems \ref{thm1} and \ref{thm2}
               to be  just a countable (non-discrete) set, then
                the result fails.  As an example one may take the model set defined by \eqref{eqmod} (with $m \geq 1$).
It is a u.d.\ set, which supports a positive measure $\mu$ whose Fourier
transform is a sum of point masses (see  \cite{mey2}), 
but  is not  contained in a finite union of translates of a lattice.

3.   It is likely that our proofs can be extended to the more general context of locally compact abelian groups. 
We do not attempt to work out the details in this paper.
            
            4.             Sometimes different approaches to mathematical models
              of quasicrystals are considered. In particular, inspired 
               by the Fibonacci sequence,
               one may look at the ``block complexity'' of a u.d.\ sequence 
               in $\R$, characterized by the number of distinct
                 blocks of  given length occurring in the sequence,
               see \cite{mey3}, \cite{bom}  and the references therein.
               It seems to be interesting to investigate the spectral
               properties of measures supported by sequences with
              ``low complexity''.


\end{document}